\documentclass[11pt]{article}

\usepackage[a4paper, total={6in, 8in}]{geometry}
\usepackage{amsmath}
\usepackage{amsfonts}
\usepackage{amssymb}
\usepackage[english]{babel}
\usepackage{amsthm}
\usepackage{tikz-cd}

\usepackage{enumitem}
\usepackage{titling}

\usepackage{xcolor}

\DeclareMathOperator{\K}{K}
\DeclareMathOperator{\Cl.K.dim}{Cl.K.dim}
\DeclareMathOperator{\C}{C}
\DeclareMathOperator{\Spec}{Spec}

\newtheorem{Theorem}{Theorem}
\newtheorem{corollary}{Corollary}[section]
\newtheorem{theorem}[corollary]{Theorem}
\newtheorem{lemma}[corollary]{Lemma}
\newtheorem{proposition}[corollary]{Proposition}
\theoremstyle{definition}
\newtheorem{definition}[corollary]{Definition}

\newtheorem{remark}[corollary]{Remark}

\title{Rings whose subrings are all Noetherian or Artinian}
\author{Nathan Blacher}
\date{}

\begin{document}

\maketitle
\vspace*{-50pt}
\begin{abstract}
We study noncommutative rings whose proper subrings all satisfy the same chain condition. We show that if every proper subring of a ring $R$ is right Noetherian, then $R$ is either right Noetherian or the trivial extension of $\mathbb{Z}$ by the Prüfer $p$-group for a prime $p$. We also prove that if every proper subring of $R$ is right Artinian, then $R$ is either right Artinian or $\mathbb{Z}$. For commutative rings, both results were proved by Gilmer and Heinzer in 1992. Our result for right Artinian subrings only generalises the absolute case of their commutative result. We generalise the full result (when only certain subrings are right Artinian) in the context of PI rings.
\end{abstract}

\section{Introduction and Notation}

Subrings naturally inherit some ring-theoretic properties, such as being a domain, the ring's characteristic and satisfying a given polynomial identity. Chain conditions however can differ widely between a ring and its subrings. For instance a field can contain a subring which is not Noetherian. In \cite{GilmerHeinzer_HereditarilyNoetherianRings}, Gilmer and Heinzer characterised commutative Noetherian rings whose subrings are all Noetherian. In \cite{GilmerHeinzer_JonssonModules}, they showed that a commutative ring whose proper subrings are all Noetherian is either Noetherian or isomorphic to the trivial extension of $\mathbb{Z}$ by the Prüfer $p$-group for a prime $p$. In the same paper they prove that a commutative ring whose proper subrings are all Artinian is either Artinian or isomorphic to $\mathbb{Z}$. We generalise both latter results to all rings, without the commutative hypothesis.

\begin{Theorem}\label{MinimalNonNoetherian}
Let $R$ be a ring which is not right Noetherian. Suppose that every proper subring of $R$ is right Noetherian. Then $R$ is the trivial extension of $\mathbb{Z}$ by the Prüfer $p$-group $\mathbb{Z}\left(p^\infty\right)$ for some prime $p$.
\end{Theorem}

The proof of Theorem~\ref{MinimalNonNoetherian} relies on the commutative version and follows the strategy of \cite{GilmerHeinzer_JonssonModules}, with the necessary changes to adapt it to the noncommutative setting.

\begin{Theorem}\label{MinimalNonArtinian}
Let $R$ be a ring which is not right Artinian. Suppose that every proper subring of $R$ is right Artinian. Then $R$ is isomorphic to $\mathbb{Z}$.
\end{Theorem}

In the commutative case, the authors of \cite{GilmerHeinzer_JonssonModules} proved the stronger result that if $C$ is a proper subring of $R$ such that every proper subring of $R$ containing $C$ is Artinian, then $R$ is Artinian. They did this by using classical Krull dimension. In particular they showed that the dimension of a commutative ring cannot strictly exceed that of all its subrings. This approach does not work in the noncommutative setting as classical Krull dimension is not as well behaved there. However in the context of rings with polynomial identity, we recover the zero-dimensional case of this result, in the form of Proposition~\ref{PIExtensionClKdim0}. This allows us to recover the stronger result for PI rings.

\begin{Theorem}\label{PIMinlRightArtinian}
Let $R$ be a PI ring and $C$ a central proper subring of $R$. Suppose that every proper subring of $R$ containing $C$ is right Artinian. Then $R$ is right Artinian.
\end{Theorem}

In \cite{GilmerOMalley_ProperSubringsNoetherian}, Gilmer and O'Malley considered (possibly noncommutative) rings without 1 whose subrings are all right Noetherian. They showed that such a ring is either right Noetherian or the zero ring on a Prüfer $p$-group. Of course in this context, the condition on subrings is much stronger than in ours, as for instance every one-sided ideal is a subring in a ring without 1.

All rings considered in this paper are associative with $1$. Subrings of a given ring $R$ must contain the identity element of $R$. We denote inclusion by $\subset$ and strict inclusion by~$\subsetneq$.

For a prime number $p$, the \textit{Prüfer $p$-group} is $\mathbb{Z}[1/p]/\mathbb{Z}$ with usual addition, denoted by $\mathbb{Z}(p^\infty)$. It is an infinite group whose proper subgroups are all finite. More precisely, every proper subgroup of $\mathbb{Z}(p^\infty)$ is cyclic and generated by $1/p^n$ for some $n\geq 0$.

For a commutative ring $C$ and a $C$-module $M$, the \textit{trivial extension} of $C$ by $M$ is the commutative ring whose additive group is the direct sum $C\oplus M$, equipped with the multiplication $(c_1,m_1)(c_2,m_2)=(c_1 c_2,c_1 m_2+c_2 m_1)$. This construction can be found in \cite[Example~2.22(A)]{Lam_Lectures}. It is sometimes called Nagata's idealisation in the literature.

Following \cite[Chapter~6]{McConnell-Robson}, we call \textit{classical Krull dimension} of a ring $R$ the maximal length of a chain of prime ideals in $R$. We denote it by $\Cl.K.dim(R)$. For a right $R$-module $M$, the \textit{Krull dimension} of $M$ is the deviation of the lattice of submodules of $M$. We denote it by $\K(M)$.

\paragraph{Acknowledgements.} The author would like to thank Vladimir Bavula for suggesting to look at the problems considered in this paper and for helpful comments and discussions. Thank you also to the anonymous reviewer whose remarks improved the presentation of the paper. The author is supported by EPSRC grant EP/T517835/1.

\section{Rings whose proper Subrings are right Noetherian}

The question we are considering was solved in the commutative case by Gilmer and Heinzer with the following result.

\begin{theorem}\cite[Corollary~4]{GilmerHeinzer_JonssonModules}\label{CommutativeMinimalNonNoetherian}
Let $R$ be a commutative ring which is not Noetherian. Suppose that every proper subring of $R$ is Noetherian. Then $R$ is the trivial extension of $\mathbb{Z}$ by $\mathbb{Z}(p^\infty)$ for some prime $p$.
\end{theorem}

We generalise this result to noncommutative rings. Our proof relies on Theorem~\ref{CommutativeMinimalNonNoetherian} and essentially follows the strategy of \cite{GilmerHeinzer_JonssonModules}, though some changes are needed to adapt it to the noncommutative world. We first need a result which is essentially a reformulation of \cite[Theorem~3.1]{GilmerOMalley_ProperSubringsNoetherian}.

\begin{proposition}\cite[Theorem~3.1]{GilmerOMalley_ProperSubringsNoetherian}\label{SquareIsZero}
Let $R$ be a ring and $I\lhd_r R$ a right ideal of $R$ that is not finitely generated. If every proper submodule of $I_R$ is finitely generated, then $I^2=0$.
\end{proposition}

\begin{proof}
Take any element $x$ in $I$. If $xI=I$ then $I=xR$ is finitely generated, contradicting our hypothesis. So the right ideal $xI$ is properly contained in $I$. Thus $xI$ is finitely generated and there are $i_1,\ldots,i_n$ in $I$ such that $xI=xi_1R+\ldots+xi_nR$. Taking any $r$ in $I$, we have $xr=xi_1r_1+\ldots+xi_nr_n$ for some $r_1,\ldots,r_n$ in $R$. Thus $x(r-i_1r_1-\ldots-i_nr_n)=0$ i.e. $r-i_1r_1-\ldots-i_nr_n$ belongs to the right annihilator of $x$ in $I$, which we call $A$. In particular, $r$ belongs to $A+i_1R+\ldots+i_nR$, and since $r$ was arbitrary in $I$, we have $I=A+i_1R+\ldots+i_nR$. Since $I_R$ is not finitely generated, $A$ cannot be finitely generated either and it follows that $A=I$, that is $xI=0$. As $x$ was arbitrary in $I$, this implies that $I^2=0$.
\end{proof}

We start by studying pairs $C\subsetneq R$, where $C$ is a central proper subring of a ring $R$, and every proper subring of $R$ containing $C$ is right Noetherian. Theorem~\ref{MinimalNonNoetherian} will then follow by taking $C$ to be the minimal subring of $R$. In the commutative case, Lemmas~\ref{SquareIn} and \ref{ACCsemiprimeIdeals} and Propositions~\ref{PrimeCaseMinlNotRightNoeth}~and~\ref{MinlNotRightNoethReducedSub} are statements of \cite[Theorem~3]{GilmerHeinzer_JonssonModules}.

\begin{lemma}\label{SquareIn}
Let $R$ be a ring and $C$ a central proper subring of $R$. Suppose that every proper subring of $R$ containing $C$ is right Noetherian. If a right ideal $I$ of $R$ is not finitely generated, then $I^2$ is contained in $(I\cap C)R$. If moreover $I^2=0$ then $R$ is commutative.
\end{lemma}

\begin{proof}
Since $C$ is Noetherian, its ideal $I\cap C$ is finitely generated, and thus so is the ideal $I':=(I\cap C)R$ of $R$. As $I_R$ is not finitely generated then neither is $I/I'$ as a right $R$-module, nor as a right ideal of $R/I'$. Therefore the ring $R/I'$ is not right Noetherian, but each of its proper subrings containing $(C+I')/I'$ is right Noetherian.

Now, since $I_R$ is not finitely generated, neither is $I$ as a right ideal of $C+I$. Thus the ring $C+I$ is not right Noetherian and it follows that $C+I=R$. Therefore as Abelian groups we have
\[R/I'=(C+I')/I'\oplus I/I'.\]
Moreover, for any right ideal $A$ of $R/I'$ properly contained in $I/I'$, we have that $(C+I')/I'\oplus A$ is a proper subring of $R/I'$. Thus $(C+I')/I'\oplus A$ is right Noetherian and it follows that $A$ is finitely generated as a right $R/I'$-module. Thus the right ideal $I/I'$ satisfies the hypothesis of Proposition~\ref{SquareIsZero} and we have $\left(I/I'\right)^2=0$, i.e. $I^2\subset I'$ as required.

Now since $R=C+I$, its commutator satisfies $[R,R]=[C+I,C+I]=[I,I]\subset I^2$ using that $C$ is central. Thus if $I^2=0$, then $R$ is commutative as required.
\end{proof}

\begin{lemma}\label{ACCsemiprimeIdeals}
Let $R$ be a ring and $C$ a central proper subring of $R$. Suppose that every proper subring of $R$ containing $C$ is right Noetherian. Then $R$ has ACC on semiprime ideals. In particular, for any ideal $A\lhd R$, there are only finitely many prime ideals minimal with respect to containing $A$.
\end{lemma}

\begin{proof}
Let $I_1\subsetneq I_2\subsetneq\ldots$ be a chain of semiprime ideals and let $I$ denote the union $\bigcup_{j\geq 1} I_j$. Since $C$ is Noetherian, $(I\cap C)R$ is finitely generated as a right ideal of $R$ and thus there exists a $k\geq 1$ such that $I_k$ contains $(I\cap C)R$. If the chain does not terminate, then $I_R$ is not finitely generated and by Lemma~\ref{SquareIn}, $I^2$ is contained in $(I\cap C)R$. Thus $I^2$ is contained in $I_k$, which is only possible when $I_k=I$, as $I_k\subset I$ and $I_k$ is semiprime. For the last statement, see for instance \cite[Proposition~4.4.13]{RowenBook}.
\end{proof}

\begin{proposition}\label{MinlNotRightNoethArtinianSub}
Let $R$ be a ring and $C$ a central proper subring of $R$. Suppose that every proper subring of $R$ containing $C$ is right Noetherian. If $C$ is Artinian then $R$ is right Noetherian.
\end{proposition}

\begin{proof}
If $R$ is commutative, this result is a consequence of \cite[Corollary~3]{GilmerHeinzer_JonssonModules}. Thus we assume that $R$ is neither commutative nor right Noetherian and we seek a contradiction. Let $I\lhd_r R$ be a right ideal of $R$ that is not finitely generated. If every proper submodule of $I$ is finitely generated then $I^2=0$ by Proposition~\ref{SquareIsZero} and hence $R$ is commutative by Lemma~\ref{SquareIn}. Therefore, since we assumed that $R$ is not commutative, it has a right ideal $I_1$ properly contained in $I$ and such that $I_1$ is not finitely generated. Similarly, there is a right ideal $I_2\subsetneq I_1$ such that $I_2$ is not finitely generated. So we can construct inductively an infinite strictly descending chain of right ideals $I=I_0\supsetneq I_1\supsetneq \ldots$ such that no $I_j$ is finitely generated. We will show it induces an infinite strictly descending chain in $C$, which is absurd since $C$ is Artinian.

Take any $j\geq 0$. Since the right ideal $I_{j+1}$ is not finitely generated, the ring $C+I_{j+1}$ is not right Noetherian and hence $C+I_{j+1}=R$. Pick an element $r$ in $I_j\setminus I_{j+1}$. Then there exist $c\in C$ and $s\in I_{j+1}$ such that $c+s=r$. Clearly, $r-s=c$ belongs to $C\cap I_j$ but does not belong to $I_{j+1}$. It follows that the proper inclusion $I_{j+1}\subsetneq I_j$ induces a proper inclusion $I_{j+1}\cap C \subsetneq I_j\cap C$. Therefore the chain $I_0\supsetneq I_1\supsetneq\ldots$ yields an infinite strictly descending chain
\[I_0\cap C\supsetneq I_1\cap C\supsetneq \ldots\] of ideals of $C$. As $C$ is Artinian, this gives the desired contradiction.
\end{proof}

\begin{proposition}\label{PrimeCaseMinlNotRightNoeth}
Let $R$ be a ring and $C$ a central proper subring of $R$. Suppose that every proper subring of $R$ containing $C$ is right Noetherian. If $R$ is prime then it is right Noetherian.
\end{proposition}

\begin{proof}
Suppose that there exists a right ideal $I\lhd_r R$ that is not finitely generated. We seek a contradiction. By Lemma~\ref{SquareIn} we have $I^2\subset (I\cap C)R$. Since $R$ is prime, $I^2\neq 0$ and therefore $I$ intersects $C$ nontrivially. So we can take a nonzero central element $x$ in $I$.

If the right ideal $xI$ is not finitely generated then $R=C+xI$. So for any $y$ in $I$, we have $y=c+xz$ for some $c$ in $C$, $z$ in $I$. In particular $c=y-xz$ belongs to $I$ and since $y$ was arbitrary in $I$ it follows that
\[I\subset I\cap C + I^2\subset (I\cap C)R\subset I, \text{ so } I=(I\cap C)R.\]
As $C$ is Noetherian, $(I\cap C)R=I$ is finitely generated, contradicting our choice of $I$.

So $xI$ is finitely generated, and we write $xI=xi_1R+\ldots+xi_nR$ for some $n\geq 1$ and $i_1,\ldots,i_n$ in $I$. Then for any $r$ in $I$ we have elements $r_1,\ldots,r_n$ of $R$ such that $xr=xi_1r_1+\ldots+xi_nr_n$. It follows that
\[0=Rx(r-i_1r_1-\ldots-i_nr_n)=xR(r-i_1r_1-\ldots-i_nr_n)\]
using that $x$ is central in $R$. Since $R$ is prime and $x$ nonzero, this implies that $r=i_1r_1+\ldots+i_nr_n$. As $r$ was arbitrary in $I$, we have $I=i_1R+\ldots+i_nR$, contradicting our choice of $I$.
\end{proof}

\begin{proposition}\label{MinlNotRightNoethReducedSub}
Let $R$ be a ring and $C$ a central proper subring of $R$. Suppose that every proper subring of $R$ containing $C$ is right Noetherian. If $R$ is not right Noetherian then the prime radical $N$ of $R$ is not Noetherian as a right $R$-module and $R=C+N$. If moreover $C$ is reduced then $R$ is the trivial extension of $C$ by $N$ and in particular $R$ is commutative.
\end{proposition}

\begin{proof}
Let $N$ be the prime radical of $R$, that is $N=\bigcap_{P\in\Spec(R)} P$. First we assume that $N_R$ is Noetherian and seek a contradiction. Since $N_R$ is Noetherian and $R$ is not right Noetherian, it follows that the ring $R/N$ is also not right Noetherian. Now, by Lemma~\ref{ACCsemiprimeIdeals}, there are only finitely many minimal primes in $R$, that is $N=P_1\cap\ldots\cap P_m$. This yields a natural inclusion of $R$-modules $R/N\hookrightarrow \bigoplus_{i=1}^m R/P_i$, and since $R/N$ is not right Noetherian, then there is at least one index $1\leq j\leq m$ such that $R/P_j$ is not right Noetherian.

So $R/P_j$ is a prime ring which is not right Noetherian. Its subring $(C+P_j)/P_j$ is Noetherian and hence proper. Moreover, every proper subring of $R/P_j$ containing $(C+P_j)/P_j$ is right Noetherian. This is impossible by Proposition~\ref{PrimeCaseMinlNotRightNoeth}, and it follows that $N_R$ is not Noetherian. In particular, the ring $C+N$ is not right Noetherian and therefore $C+N=R$.

Now in the case where $C$ is reduced, we obtain $C\cap N=0$ as $N$ is nil (see for instance \cite[Theorem~0.2.6]{McConnell-Robson}). Thus $R=C\oplus N$. If $M$ is a proper submodule of $N_R$, then $C\oplus M$ forms a proper subring  of $R$. In particular, $C\oplus M$ is right Noetherian and hence $M_R$ is finitely generated. Thus, by Proposition~\ref{SquareIsZero}, $N^2=0$. Now take any two elements $r_1,r_2\in R=C\oplus N$ and write $r_1=c_1+n_1$ and $r_2=c_2+n_2$, for $c_1,c_2\in C$ and $n_1,n_2\in N$. Then
\[r_1r_2=(c_1+n_1)(c_2+n_2)=(c_1c_2)+(c_1n_2+c_2n_1).\]
It follows that $R$ is ring isomorphic to the trivial extension of the commutative ring $C$ by the $C$-module $N$. In particular $R$ is commutative.
\end{proof}

We can now prove Theorem~\ref{MinimalNonNoetherian}, that is we generalise Theorem~\ref{CommutativeMinimalNonNoetherian} to the noncommutative setting.

\begin{proof}[Proof of Theorem~\ref{MinimalNonNoetherian}]
Let $R$ be a ring which is not right Noetherian and whose proper subrings are all right Noetherian. Let $C$ be the minimal subring of $R$. By Proposition~\ref{MinlNotRightNoethArtinianSub}, the ring $C$ cannot be Artinian, i.e. $R$ has characteristic zero. Then $C$ is equal to $\mathbb{Z}$, which is reduced. Thus $R$ is commutative by Proposition~\ref{MinlNotRightNoethReducedSub}. We are done by Theorem~\ref{CommutativeMinimalNonNoetherian}.
\end{proof}

\begin{remark}
We have shown that, except for the trivial extensions of $\mathbb{Z}$ by a Prüfer $p$-group $\mathbb{Z}\left(p^\infty\right)$, a ring whose proper subrings are right Noetherian must be right Noetherian itself. Following \cite{GilmerHeinzer_HereditarilyNoetherianRings}, such rings can be called \textit{hereditarily right Noetherian} and it would be interesting to describe them. They include for instance any ring which is a finitely generated $\mathbb{Z}$-module (i.e. whose additive group is finitely generated). Naturally, the centre of a hereditarily right Noetherian ring is a commutative hereditarily Noetherian ring, described in \cite{GilmerHeinzer_HereditarilyNoetherianRings}.
\end{remark}

\section{Rings whose proper Subrings are right Artinian}

The question was also answered for the commutative Artinian property by Gilmer and Heinzer, with the following result.

\begin{theorem}\label{CommutativeHereditaryArtinian}\cite[Theorem~2]{GilmerHeinzer_JonssonModules}
Let $R$ be a commutative ring and $C$ a proper subring of $R$. Suppose that every proper subring of $R$ containing $C$ is Artinian. Then $R$ is Artinian.
\end{theorem}

Taking $C$ to be the minimal subring of $R$ immediately yields the following result.

\begin{corollary}\cite[Corollary~2]{GilmerHeinzer_JonssonModules}\label{CommutativeMNA}
Let $R$ be a commutative ring which is not Artinian. Suppose that every proper subring of $R$ is Artinian. Then $R$ is isomorphic to $\mathbb{Z}$.
\end{corollary}

In the noncommutative context, we start by noting that the Noetherian case yields the following immediate consequence.

\begin{lemma}\label{MinimalNonArtinianIsNoetherian}
Suppose that every proper subring of a ring $R$ is right Artinian. Then $R$ is right Noetherian.
\end{lemma}

\begin{proof}
Every proper subring of $R$ is right Artinian, and hence right Noetherian. By Theorem~\ref{MinimalNonNoetherian}, $R$ is either right Noetherian or the trivial extension of $\mathbb{Z}$ by $\mathbb{Z}\left(p^\infty\right)$ for some prime $p$. But this trivial extension strictly contains $\mathbb{Z}$ as a subring, which is not Artinian. It follows that $R$ is right Noetherian.
\end{proof}

We now prove Theorem~\ref{MinimalNonArtinian}, that is we generalise Corollary~\ref{CommutativeMNA} to noncommutative rings.

\begin{proof}[Proof of Theorem~\ref{MinimalNonArtinian}]
Let $R$ be a ring which is not right Artinian and whose proper subrings are all right Artinian. We assume that $R$ is not $\mathbb{Z}$ and seek a contradiction. If the characteristic of $R$ is zero, then $R$ contains a copy of $\mathbb{Z}$ as a proper subring, which contradicts our hypothesis. Therefore $R$ has finite characteristic. By Lemma~\ref{MinimalNonArtinianIsNoetherian}, $R$ is right Noetherian. Take a two-sided ideal $I$ of $R$ maximal with respect to $R/I$ not being right Artinian. Then the ring $R/I$ satisfies the hypothesis of the theorem. So without loss of generality we can take $R$ to be $R/I$. Furthermore, every proper homomorphic image of $R/I$ is right Artinian. We will show this yields a contradiction.

We first show that $R$ has prime characteristic. Write $p_1^{l_1}\cdots p_n^{l_n}$ for the prime factorisation of the characteristic of $R$. We have a ring isomorphism between the minimal subring $\mathbb{Z}/(p_1^{l_1}\cdots p_n^{l_n})\mathbb{Z}$ and $\Pi_{i=1}^n\mathbb{Z}/p_i^{l_i}\mathbb{Z}$. The preimage of $(1,0,\ldots,0)$ under this isomorphism gives us a nonzero central idempotent $e$ in $R$. Thus $R=Re\oplus R(1-e)$ as a direct sum of ideals. If $e\neq 1$, then both $R/Re\cong R(1-e)$ and $R/R(1-e)\cong Re$ are proper homomorphic images of $R$, and hence Artinian as right $R$-modules. In this case, $R_R$ is a direct sum of Artinian modules and is hence right Artinian as a ring, contradicting our hypothesis. Therefore $e=1$, that is the characteristic of $R$ is a prime power $p^l$.

Now consider the minimal subring $\mathbb{Z}/p^l\mathbb{Z}$ of $R$. If $R=\mathbb{Z}/p^l\mathbb{Z}+pR$ then $R=\mathbb{Z}/p^l\mathbb{Z}+p(\mathbb{Z}/p^l\mathbb{Z}+pR)=\mathbb{Z}/p^l\mathbb{Z}+p^2R$ and repeating the substitution yields $R=\mathbb{Z}/p^l\mathbb{Z}+p^lR=\mathbb{Z}/p^l\mathbb{Z}$, which is impossible as $R$ is not right Artinian. Thus $\mathbb{Z}/p^l\mathbb{Z}+pR$ is a proper subring of $R$. In particular it is right Artinian. Therefore the right $(\mathbb{Z}/p^l\mathbb{Z}+pR)$-module $pR$ is Artinian, and it follows that $pR$ is Artinian as a right $R$-module. The short exact sequence
\[0\to pR\to R\to R/pR\to 0\]
then implies that $R/pR$ is not right Artinian. This is only possible if $R/pR$ is not a proper homomorphic image of $R$, i.e. if $pR=0$. Therefore the characteristic of $R$ is $p$.

Now, $R$ is an $F_p = \mathbb{Z}/p\mathbb{Z}$-algebra which is not right Artinian and such
that every proper subring is right Artinian. Take an infinite descending chain $R \supsetneq A_1 \supsetneq
A_2 \supsetneq\ldots$ of right ideals of $R$. The right $R$-module $A_2$ is not Artinian, so the ring $F_p + A_2$
is not right Artinian, hence $F_p + A_2 = R$. Take $x$ in $A_1\setminus A_2$. Then $x = u + a$ for some $a$ in
$A_2$ and nonzero $u$ in $F_p$. But $u=x-a$ belongs to $A_1$, which is impossible as $F_p$ is a field
and $A_1$ is a proper ideal. This is the required contradiction, and it follows that $R$ is isomorphic to $\mathbb{Z}$.
\end{proof}

The proof of \cite[Theorem~2]{GilmerHeinzer_JonssonModules} uses classical Krull dimension. It is not as well behaved in the general noncommutative setting, and our proof of Theorem~\ref{MinimalNonArtinian} followed a different approach. In the next section, we recover some of the results of \cite{GilmerHeinzer_JonssonModules} in the context of rings satisfying a polynomial identity (PI).

\section{Classical Krull Dimension 0 in PI rings}

\begin{definition}\cite[Definition~1.7.27]{RowenBook}
Let $R$ be a ring, and $W$ a subring of $R$. Call $R$ an \textit{extension} of $W$ if $R=W\C_R(W)$, where $\C_R(W)$ is the centraliser of $W$ in $R$, i.e. elements of $R$ that commute with $W$.
\end{definition}

Of course, every ring is an extension of any of its central subrings. We will use the following fact repeatedly.

\begin{remark}\cite[Remark 1.9.5]{RowenBook}
Let $R$ be an extension of a subring $W$. If $P$ is a prime ideal of $R$, then $W\cap P$ is a prime ideal of $W$.
\end{remark}


For $r\in R$ we denote by $W(r)$ the subring of $R$ generated by $W$ and $r$. The following is a generalisation to PI rings of the case $n=0$ of \cite[Theorem~1]{GilmerHeinzer_JonssonModules}.

\begin{proposition}\label{PIExtensionClKdim0}
Let $R$ be a PI ring that is an extension of a proper subring $W$. Suppose that every proper subring of $R$ of the form $W(r)$ has classical Krull dimension zero. Then the classical Krull dimension of $R$ is also zero.
\end{proposition}

\begin{proof}
Let us assume that there are prime ideals $P_0\subsetneq P_1$ in $R$. We seek a contradiction. Since $R/P_0$ is a prime PI ring, the nontrivial ideal $P_1/P_0$ contains a nonzero central element, $\bar{c}$ say (see for instance \cite[Theorem~1.6.27]{RowenBook}). Now, $R$ is an extension of $W$ so $R/P_0$ is an extension of $W':=(W+P_0)/P_0$. Moreover since $\bar{c}$ is central, $R/P_0$ also is an extension of $W'\left(\bar{c}\right)$. In particular, the prime ideal $P_1/P_0$ induces a prime ideal $P_1':=(P_1/P_0)\cap W'\left(\bar{c}\right)$ of $W'\left(\bar{c}\right)$. The ring $W'\left(\bar{c}\right)$ is prime, and its prime ideal $P_1'$ is nonzero as it contains $\bar{c}\neq 0$. Thus the classical Krull dimension of $W'\left(\bar{c}\right)$ is at least one.

Taking $c$ in the preimage of $\bar{c}$ under the natural projection $R\twoheadrightarrow R/P_0$, we have $\Cl.K.dim W(c)\geq \Cl.K.dim W'\left(\bar{c}\right)\geq 1$. It follows by hypothesis that $W(c)=R$. We can show similarly that $W\left(c^2\right)=R$. In particular, $c$ lies in $W\left(c^2\right)$ and therefore $\bar{c}$ belongs to $W'\left(\bar{c}^2\right)$. Thus, using that $\bar{c}$ is central, it must be a root of a nonzero polynomial with coefficients in $W'$. Let $l\geq 1$ be minimal such that there exist $w_0,\ldots,w_l\in W'$ not all zero and satisfying $\sum_{i=0}^l w_i\bar{c}^i=0$. Then $-w_0=\bar{c}\left(\sum_{i=0}^{l-1}w_{i+1}\bar{c}^i\right)$ belongs to $P_1/P_0$, as $\bar{c}$ does. Moreover, since $\bar{c}$ is a central element in the prime ring $R/P_0$, it must be regular. Then, by minimality of $l$, it follows that $w_0$ is not zero and thus the prime ideal $W'\cap (P_1/P_0)$ is nontrivial. Therefore
\[1\leq \Cl.K.dim W'=\Cl.K.dim W/(W\cap P_0)\leq \Cl.K.dim W =0\]
which is absurd, showing that the classical Krull dimension of $R$ is zero.
\end{proof}

\begin{proposition}\label{PIdim0ArtinianIff}
Let $R$ be a PI ring with classical Krull dimension zero. The following are equivalent:
\begin{enumerate}[topsep=0pt]
\item[(i)]$R$ is right Artinian;
\item[(ii)]the right Krull dimension of $R$ exists;
\item[(iii)]every ideal of $R$ is finitely generated as a right $R$-module.
\end{enumerate}
\end{proposition}

\begin{proof}
That (i) implies (ii) and (iii) is trivial. Now if $R$ has right Krull dimension $\K(R_R)$, then $\K(R_R)=\sup\lbrace \K(R/P):P\in\Spec(R)\rbrace$ (see for instance \cite[Corollary 6.3.8]{McConnell-Robson}). This is zero as every prime ideal of $R$ is maximal, and the simple PI rings $R/P$ are Artinian by Kaplansky's Theorem. Thus (ii) implies (i).

Finally, we prove that (iii) implies (i). We assume that $R$ is not right Artinian, and seek a contradiction. Since $R$ has ACC on ideals, we can take an ideal $P$ of $R$ maximal with respect to $R/P$ not being right Artinian. Let $A,B$ be ideals strictly containing $P$. Since $A_R$ is finitely generated, it follows that $A/AB$ is finitely generated as a right module over $R/B$. As $R/B$ is right Artinian, this implies that $A/AB$ is Artinian as a right $R$-module. Thus, the short exact sequence of right $R$-modules $A/AB\hookrightarrow R/AB \twoheadrightarrow R/A$ implies that the ring $R/AB$ is right Artinian. Therefore $AB$ is not contained in $P$, and we have shown that $P$ is prime. The classical Krull dimension of $R$ is zero so $P$ is maximal, and $R/P$ is a simple PI ring, which is Artinian by Kaplansky's Theorem, contradicting our choice of $P$.
\end{proof}

Lastly, we prove Theorem~\ref{PIMinlRightArtinian}, i.e. we generalise Theorem~\ref{CommutativeHereditaryArtinian} to PI rings.

\begin{proof}[Proof of Theorem~\ref{PIMinlRightArtinian}]
Let $R$ be a PI ring and $C$ a central proper subring of $R$ such that every proper subring of $R$ containing $C$ is right Artinian. Then every proper subring of $R$ containing $C$ has classical Krull dimension zero. Since $R$ is an extension of $C$, this implies that the classical Krull dimension of $R$ is also zero by Proposition~\ref{PIExtensionClKdim0}. By Proposition~\ref{MinlNotRightNoethArtinianSub}, $R$ is right Noetherian. Thus $R$ satisfies condition (iii) of Proposition~\ref{PIdim0ArtinianIff} and it follows that $R$ is right Artinian.
\end{proof}
\bibliography{Rings_Whose_Subrings_are_all_Noetherian_or_Artinian}

\begin{thebibliography}{1}

\bibitem{GilmerHeinzer_HereditarilyNoetherianRings}
R.~Gilmer and W.~Heinzer.
\newblock Noetherian pairs and hereditarily {N}oetherian rings.
\newblock {\em Arch. Math. (Basel)}, 41(2):131--138, 1983.

\bibitem{GilmerHeinzer_JonssonModules}
R.~Gilmer and W.~Heinzer.
\newblock An application of {J}\'onsson modules to some questions concerning
  proper subrings.
\newblock {\em Math. Scand.}, 70(1):34--42, 1992.

\bibitem{GilmerOMalley_ProperSubringsNoetherian}
R.~Gilmer and M.~O'Malley.
\newblock Non-{N}oetherian rings for which each proper subring is {N}oetherian.
\newblock {\em Math. Scand.}, 31:118--122, 1972.

\bibitem{Lam_Lectures}
T.~Y. Lam.
\newblock {\em Lectures on modules and rings}, volume 189 of {\em Graduate
  Texts in Mathematics}.
\newblock Springer-Verlag, New York, 1999.

\bibitem{McConnell-Robson}
J.~C. McConnell and J.~C. Robson.
\newblock {\em Noncommutative {N}oetherian rings}, volume~30 of {\em Graduate
  Studies in Mathematics}.
\newblock American Mathematical Society, Providence, RI, revised edition, 2001.

\bibitem{RowenBook}
L.~H. Rowen.
\newblock {\em Polynomial identities in ring theory}, volume~84 of {\em Pure
  and Applied Mathematics}.
\newblock Academic Press, Inc., New York-London, 1980.

\end{thebibliography}
\bibliographystyle{abbrv}

\begin{minipage}{0.5\textwidth}
School of Mathematical and Physical Sciences\\
University of Sheffield\\
Hicks Building\\
Sheffield S3 7RH\\
United Kingdom\\
email: nlblacher@sheffield.ac.uk
\end{minipage}
\end{document}